\documentclass[12pt]{amsart}
\usepackage{amsmath,amssymb}
\usepackage{latexsym}
\usepackage{amsthm}
\usepackage[normalem]{ulem}
\usepackage{lineno}
\usepackage{paralist,enumitem}
\usepackage{hyperref}
\usepackage{color}
\definecolor{darkgreen}{rgb}{0,0.6,0}

\definecolor{darkblue}{rgb}{0,0,0.9}

\title[Primitive permutation groups of affine type]{Primitive permutation groups of finite Morley rank and affine type}
\author{Ay\c{s}e Berkman}
\address{Mathematics Department, Mimar Sinan Fine Arts University, Istanbul, Turkey}
\email{ayse.berkman@msgsu.edu.tr, ayseberkman@gmail.com}
\author{Alexandre Borovik}
\address{Department of Mathematics, University of Manchester, UK}
\email{alexandre $\gg {\rm at} \ll $ borovik.net}
\thanks{\copyright\ 2024 Ay\c{s}e Berkman and Alexandre Borovik}

\date{6 December 2024}

\begin{document}

	\newtheorem{lemma}{Lemma}[section]
	\newtheorem{theorem}{Theorem}
	\newtheorem{corollary}[lemma]{Corollary}
	\newtheorem{proposition}[lemma]{Proposition}
	\newtheorem{remark}[lemma]{Remark}
	\newtheorem{example}[lemma]{Example}
	\newtheorem{fact}[lemma]{Fact}
    \newtheorem{our}[lemma]{Theorem}
	\newtheorem{question}[lemma]{Question}
	\theoremstyle{definition}
	\newtheorem{definition}{Definition}
	\newtheorem{conjecture}{Conjecture}
\newtheorem{problem}[conjecture]{Problem}
	\newtheorem{examples}[lemma]{Examples}
		\newtheorem{hypothesis}{Hypothesis}
	\newcommand{\acf}{algebraically closed field }
	\newcommand{\acfd}{algebraically closed field}
	\newcommand{\acfs}{algebraically closed fields}
	\newcommand{\fmr}{finite Morley rank }
	\newcommand{\rk}{\operatorname{rk}}
	\newcommand{\ddeg}{\operatorname{deg}}
	\newcommand{\psrk}{{\rm psrk}}
	\newcommand{\stab}{{\rm stab}}
	\newcommand{\symm}{{\rm Sym}}
	\newcommand{\fmrd}{finite Morley rank}
	\newcommand{\sll}{\operatorname{SL}}

\def\nn{\ensuremath \mathbb{N}}

\newcommand{\bq}{\begin{quote}}
\newcommand{\eq}{\end{quote}}
\newcommand{\bd}{\begin{description}}
\newcommand{\ed}{\end{description}}
\newcommand{\bi}{\begin{itemize}}
\newcommand{\ei}{\end{itemize}}
\newcommand{\ben}{\begin{enumerate}}
\newcommand{\een}{\end{enumerate}}

\begin{abstract}
We give a review of one of the lines in development of the theory of groups of finite Morley rank. These groups naturally appear in model theory as model-theoretic analogues of Galois groups, therefore their actions and their role as permutation groups is of primary interest. We restrict our story to the study of connected groups of finite Morley rank $G$ acting in a definably primitive way on a set $X$ and containing a definable abelian normal subgroup $V$ which acts on $X$ regularly --- the so-called \emph{primitive groups of affine type}. For reasons explained in the paper, this case plays a central role in the theory.
\end{abstract}
	
\maketitle

\section{Introduction}
This paper is a description of a particular research project in  model-theoretic algebra which, we hope, is interesting as a case study in project design and project planning in mathematical research. The authors (who refer to themselves in this paper as ``we''), worked on it for  13 years starting in 2011. Our papers from this project,  if one reads them in the chronological order \cite{bbgeneric,bbpseudo,bbsharp,bbnotsharp,avb,bbsolvable}, give no idea of what our project was aimed at, and why it was developed this particular way. In the present paper, we are trying to explain {our thinking  and  strategy}.

We are trying to keep our paper accessible and avoid unnecessary technical details. Here is the main result of the series of six {papers} \cite{bbgeneric,bbpseudo,bbsharp,bbnotsharp,avb,bbsolvable}:

\medskip

\noindent
\textbf{Main Theorem} (Theorem \ref{th:mail-on-primitive-affine}, page  \pageref{th:mail-on-primitive-affine}).	\emph{Let $(G,X)$ be a connected definably primitive permutation group of \fmr and  affine type. Assume that the action is generically $t$-transitive and $\rk(X) = n$.}

\emph{Then $t \leqslant n +1$.}

\emph{Moreover, if $t = n +1$, then  the group $G$ is isomorphic to the affine general linear group $\operatorname{\rm AGL}_n(K) \simeq K^n\rtimes \operatorname{\rm GL}_n(K)$  for some \acf $K$,  and the action of $G$ on $X$ is equivalent to $\operatorname{\rm AGL}_n(K)$  acting on the $n$-dimensional affine space $X=\mathbb{A}_n(K)$.}

\medskip

All the terms involved will soon be explained.

\subsection{Some of the very first definitions and a historic remark}
Let $(G,X)$ be a permutation group (that is, a group $G$ acts on a set $X$  and the kernel of the action is trivial); both $G$ and $X$ could be infinite. One says that $(G,X)$ is \emph{primitive}, if the only $G$-invariant equivalence relations on $X$ are trivial or universal.

A permutation group is \emph{regular} (or \emph{sharply transitive}) if it is transitive and the stabiliser of each point is trivial. A primitive permutation group is of \emph{affine type} if it contains a normal abelian regular subgroup. In the finite case, primitive solvable groups are inevitably of affine type, and they had been introduced and studied by \'{E}variste Galois, in a work which also led to  a new class of algebraic objects---finite fields, nowadays known as Galois fields. A very illuminating analysis of Galois's work on primitivity is Peter Neumann's paper \cite{Neumann2006}.

\subsection{Bases of a permutation group.} Given a permutation group $G$ (not necessary finite) on a set $X$, a finite subset of $X$ is said to be a \emph{base} for $G$ if its pointwise stabilizer in $G$ is trivial. The minimal size of a base for $G$ on $X$ is denoted by $b(G,X)$. Bases have been originally introduced  for computation in finite groups as they provide, for many groups $G$, convenient and compact parametrisations: if $b_1, \dots, b_n$ is a base, then every element  $g\in G$ is uniquely determined by the $n$-tuple $(b_1^g,\dots, b_n^g)$. Alternative ${\rm Alt}_n$ and symmetric ${\rm Sym}_n$ groups are very special extreme cases:
\[
b({\rm Alt}_n, \{1,\dots, n\})= n-2 \mbox{ and } b({\rm Sym}_n, \{1,\dots, n\}) = n-1,
\]
there is no gain in efficiency of calculations from the use of bases.

Notice that the concept of `parametrisation' of an algebraic structure in terms of other structures belongs to  model theory, and is one of the reasons for deep conceptual links and analogies between  finite group theory and group-theoretic aspects of  model theory.

\subsection{Binding groups} Binding groups are model-theoretic analogues of Galois groups; they were introduced in the 1970s by Boris Zilber \cite{Zi-binding}. They are very natural mathematical objects, as the following canonical example shows.

Let $X$ be a finite dimensional vector space and $G$ its automorphism group. A  basis of $X$ as a vector space is also a base for the action of $G$ as a permutation group on $X$, and elements of $G$ can be encoded as images of this base, that is, finite tuples of vectors, commonly known as matrices; the dependencies between vectors in the space allow us to express composition of automorphisms as product of matrices: the group $G$ becomes \emph{definable} in $X$. It was Zilber's seminal discovery that a similar construction can be carried out over a vast class of mathematical structures which have good logical properties of `dependency' between elements (although the nature of this dependency could be far away from that of  linear dependency in vector spaces).

\subsection{Groups of finite Morley rank} \emph{Groups of finite Morley rank} emerged as binding groups in Zilber's analysis of arbitrary $\aleph_1$-categorical structures; this made them a focal point of model theoretic algebra. They are abstract groups equipped with a suitable notion of dimension called \emph{Morley rank} (denoted $\rk$) on definable sets.

Examples of groups of finite Morley rank are furnished by algebraic groups over algebraically closed fields, with Morley rank equal to the usual algebraic-geometric dimension of Zariski closures. The theory of groups of finite Morley rank started in pioneering works by Zilber \cite{Zi-GR}  and Cherlin \cite{Ch-GSR}; they formulated what remains the central conjecture in the subject:
\begin{conjecture} \label{conj:Cherlin-Zilber} \emph{The Cherlin-Zilber Algebraicity Conjecture.}
Infinite simple groups of finite Morley rank are simple algebraic groups over algebraically closed fields.
\end{conjecture}

The theory of groups of finite Morley rank has a simple and transparent axiomatisation which could be found in books \cite{abc,bn}. We have a \emph{definable universe} $\mathcal{U}$, a set of sets closed under standard set-theoretic operations of union, intersection, set-theoretic difference, finite direct products, projections. Finite subsets of definable sets are definable. Non-empty definable sets are assigned non-negative integers, their ranks.\footnote{The axioms for a definable universe allow multiplying  of the rank function by a positive  integer constant. But, because of the nature of results in this paper, the expression  “rank 1”, as in Theorem \ref{fact-pseudo},  is actually meaningful, because increase in value sof the rank function increases the strength of other assumptions in our results.  } A function is definable if its graph is definable. A group is definable if its ground set is definable and operations of multiplication and inversion are definable as functions. In particular, finite groups are groups of Morley rank $0$.

All these definitions and axioms are chosen in a way making them valid for constructible sets over an algebraically closed field, with the Morley rank being the  algebraic-geometric dimension of Zariski closures. Once stated at the beginning of the book \cite{abc}, the axioms are hardly mentioned again in 500+ pages --- the analogy with elementary level algebraic geometry is sufficient for understanding. In case of any difficulties with understanding this text, take a glance at a few first chapters of \cite{abc} or \cite{bn}.

From now on, when we talk about a permutation group $(G,X)$ of \fmr we assume that the group $G$, the set $X$, and the action of $G$ on $X$ all belong to some definable universe.

\subsection{Some elementary facts} A few basic properties of groups of finite Morley rank need to be mentioned from the beginning because they justify subsequent definitions. Stabilisers of points in definable actions are definable (in particular, centralisers of elements are definable); there is a descending chain condition for definable subgroups, which means that definable actions on sets have finite bases. It also follows that every group  $G$ of  \fmr contains a unique minimal definable normal subgroup $G^\circ$ of finite index, called the \emph{connected component} of $G$. A group $G$ is \emph{connected} if $G^\circ = G$. \emph{Simple groups} of \fmr are those which contain no non-trivial proper definable normal subgroups; it could be shown that they are simple as abstract groups, that is, has no non-trivial proper normal subgroups, definable or not. A compact exposition of all introductory material can be found in \cite{abc}.

We recommend Gregory Cherlin's informative and incisive  survey \cite{Cherlin2023} of the current state of the classification of simple groups of finite Morley rank. In short, there are  three types of   simple  groups of finite Morley rank: \emph{degenerate} (they do not contain involutions), \emph{odd} (contain involutions, but do not contain infinite subgroups of  exponent $2$), \emph{even} (contain infinite subgroups of exponent $2$).  Simple groups of even type have been identified as simple algebraic groups over algebraically closed fields of  characteristic $2$; a self-contained proof fills a book of 560 pages \cite{abc}. Little is known about simple groups of degenerate type beyond a fantastic result by Olivier Fr\'{e}con on groups of Morley rank 3 \cite{Frecon2018} (beautifully elucidated by Luis-Jaime Corredor and Adrien Deloro \cite{Corredor-Deloro2022}). On the contrary, quite a lot is known about groups of odd type, but still not enough for proving for them the special case of  Conjecture \ref{conj:Cherlin-Zilber}:

\begin{conjecture}  \emph{Cherlin-Zilber Algebraicity Conjecture for Groups of Odd Type}.
Infinite simple groups of finite Morley rank and odd type are simple algebraic groups over algebraically closed fields of odd or zero characteristic.
\end{conjecture}

\subsection{Permutation groups of finite Morley rank} Talking about permutation actions of groups of finite Morley ranks, we mean \emph{definable} actions on \emph{definable} sets within some ambient definable universe.

The following simple lemma links the sizes of the bases of a permutation group $(G,X)$ with the ranks of $G$ and $X$; a proof is given as a typical example of arguments in our theory.

\begin{lemma} Let $B = (b_1,\dots, b_k)$ be a base of a permutation group $(G,X)$ of \fmr. Then
\[
\rk(G) \leqslant k \cdot \rk(X).
\]
\end{lemma}

\begin{proof}
If we consider the action of $G$ on the direct power $X^k$ and the $G$-orbit $B^G$ of $B\in X^k$ then the map $g\mapsto B^g$ from $G$ to $B^G$ is injective and therefore
\[
\rk(G) = \rk(B^G) \leqslant \rk(X^k) = k \cdot \rk(X).
\]
\end{proof}

The theory we are talking about is woven of this kind of little observations, but as a whole it is very deep and structurally complex.

For permutation groups of \fmrd,  Borovik and Cherlin have demonstrated the  power of several classification techniques --- adaptation of methods of CFSG (Classification of Finite Simple Groups), classification of groups of finite Morley rank and even type \cite{abc}, specific  techniques  for groups of odd type \cite{BoBu,BBC-ID,Deloro-Jaligot}, and analogues of O'Nan-Scott/Aschbacher reductions \cite{Macpherson-Pillay95} --- to analyze basic questions about actions. However, many natural questions about permutation groups of finite Morley rank are still open. We reiterate:
It is a good understanding of permutation  groups of finite Morley rank, applied to binding groups, that would matter for the wider model theory.
This thesis is confirmed, for example, in the recent papers of  Freitag, Jimenez, and Moosa \cite{FM, FJM}.

\subsection{Difficulties arising at the very first step} The following example, well known to algebraic geometers, shows striking differences between the theories of finite permutation groups and permutation groups of finite Morley rank.

\begin{example} \label{ex:unlimited} Let $K$ be an algebraically closed field and  $$G= K
\oplus \cdots \oplus K$$  a direct sum of $n$ copies of the additive
group of $K$. Consider the following action of $G$ on the affine space
${\mathbb{A}}^2(K)$:
$$
(a_1, \ldots, a_n): \; \left( \begin{array}{c}
                           x \\ y
                       \end{array} \right) \; \mapsto \;
                       \left( \begin{array}{c}  x \\ y + a_1x + \cdots +
a_nx^n \end{array} \right)
$$
Obviously this  action  is faithful.
\end{example}

We see that Morley ranks of  groups acting faithfully and definably on a set of rank $2$ are unbounded. Of course, actions in the example above are not transitive. But a slightly more sophisticated  construction \cite[Example 8]{borche} produces groups of unbounded finite Morley rank  acting definably, faithfully, and \emph{transitively} on a set of Morley rank 2.

This makes the study of permutation groups of finite Morley rank much more difficult than the study of finite permutation groups $(G,X)$, where we have the well-known bound $|G| \leqslant |X|!$. It is worth remembering that the classification of the critically important subclass of finite primitive groups:  $2$-transitive permutation groups became possible only as a corollary of the enormous Classification of Finite Simple Groups.

Also, we have to keep in mind that permutation groups of \fmr include algebraic groups over \acfs\ acting by birational automorphisms on algebraic varieties. As one example, the study of algebraic subgroups of Cremona groups remains a hot topic, see \cite{Fanelli-Floris-Zimmerman2023,Schneider-Zimmermann2021} and the bibliography contained there.

In the spirit of the Cherlin-Zilber Algebraicity Conjecture, we hope to reduce the study of permutation groups of \fmr to the case of definable actions of algebraic groups.

\subsection{Primitive permutation groups}

However,  in the case of \emph{primitive} actions the situation is better than Example \ref{ex:unlimited} suggests:

\begin{theorem} \emph{(}Borovik and Cherlin \cite{borche}\emph{)} \label{th:Borovik-Cherlin} There is a function $$f: \nn \rightarrow \nn$$ such that if a group $G$ of finite Morley rank acts  on a definable set $X$ definably, faithfully, and definably primitively \emph{(}i.e., preserves no definable proper nontrivial equivalence relation\emph{)}, then
\[
{\rm rk}(G) \leqslant f({\rm rk}(X)).
\]
\end{theorem}

In finite group theory, a good understanding of primitive permutation groups (and its profound consequences for combinatorics, representation theory, and model theory) became possible only as a result of the forbiddingly difficult CFSG. In the  finite Morley rank context, we expect that at least applications to model theory can be achieved using  powerful techniques already developed for the analysis  of simple groups of finite Morley rank.

Theorem \ref{th:Borovik-Cherlin} was the starting point and motivation for our research project described in this paper.

\subsection{Multiple generic transitivity} This concept played a central role in the proof of Theorem \ref{th:Borovik-Cherlin}.

 When $X$ has \fmrd,  a definable subset $A\subseteq X$ is called {\em generic} in $X$, if $\rk(X\smallsetminus A)<\rk(X)$.
	
Let $G$ be a group of \fmr acting definably on a set $X$ of \fmrd. If the induced action of $G$ on $X^n$ is (sharply) transitive on a generic subset  {of $X^n$}, then we say $G$ acts {\em generically \emph{(}sharply\emph{)} $n$-transitively} on $X$.

 For algebraic groups, this concept was introduced by Vladimir Popov \cite{popov} for connected algebraic groups $G$ acting  {rationally} on an irreducible  {algebraic} variety $V$: if the induced action of $G$ on $V^n$ is transitive on an open subset  {of $V^n$}, then the action is called {\em generically} $n$-transitive.

In characteristic $0$, Popov computed maximal generic transitivity degrees of actions of simple algebraic groups on irreducible algebraic varieties:

 \begin{fact}{\rm (Popov \cite{popov})} \label{fact:Popov} If the characteristic of the underlying field is $0$, then simple algebraic groups have the following maximal degrees of generic transitivity:
\[\begin{array}{c|c|c|c|c|c|c|c|c}
A_n & B_n, n\geq 3 & C_n, n\geq 2 & D_n, n\geq 4 & E_6 & E_7 & E_8 & F_4 & G_2\\\hline
n+2 & 3 & 3 & 3 & 4 & 3 & 2 & 2 & 2
\end{array}\]
 \end{fact}

\begin{examples} \cite{borche} \label{ex:multiple-generic-transitivity}
 For every $m\geqslant 1$ and \acf $K$, the natural action of:
\begin{compactenum}
	\item  $(K^*)^n$ as  {a group of diagonal matrices  on the vector space $K^n$} is generically sharply  transitive.

    \item {The special linear group} $\operatorname{SL}_n(K)$ on  {the vector space} $K^n$ is generically (but not sharply) $(n-1)$-transitive.

	\item  {The general linear group} $\operatorname{GL}_n(K)$ on  {the vector space} $K^n$ is generically sharply $n$-transitive.

    \item  {The affine special linear group}  {$\operatorname{ASL}_n(K) = K^n \rtimes \operatorname{SL}_n(K)$} on $K^n$ is generically (but not sharply)  $n$-transitive.
	
	\item  {The affine general linear group}  {$\operatorname{AGL}_n(K) = K^n \rtimes \operatorname{GL}_n(K)$} on $K^n$ is generically sharply $(n+1)$-transitive.
	
	\item  {The projective general linear group}  $\operatorname{PGL}_{n+1}(K)$ on  {the projective space} $\mathcal{P}_{n}(K)$ is generically sharply $(n+2)$-transitive.
	
\end{compactenum}
\end{examples}

Notice that in cases (1), (2), and (3) the actions are not primitive because they fix $0 \in K^n$ hence not transitive.  After being restricted to $K^n \smallsetminus \{0\}$, these actions are still not primitive because the orbits of the group of scalars provide an invariant partition. Actions (4), (5), and (6) are primitive.

Adopting the notation from \cite{borche}, denote by $\rho(n)$ the maximum possible rank of a group acting definably primitively on a set of Morley rank $n$, and by $\tau(n)$ the maximum possible degree of generic transitivity of a definably primitive action of a group on a set of Morley rank $n$.

The following Proposition adapted from \cite[Proposition 2.3]{borche}  explains the role of multiple generic transitivity in estimates for ranks of primitive groups.
	
\begin{proposition} In the above setting, we have $$n\tau(n)\leqslant\rho(n)\leqslant n\tau(n)+{n\choose 2}.$$
\end{proposition}

Examples \ref{ex:multiple-generic-transitivity}  suggest a tight bound for the degree of generic multiple transitivity:

\begin{conjecture} \emph{\cite[Conjecture 2]{bordel}}
\[
\tau(n) = n+2.
\]
\end{conjecture}

\subsection{The Macpherson-Pillay classification  of primitive groups of \fmr} 
\label{sec:Macpherson-Pillay}
Dugald Macpherson and Anand Pillay \cite{Macpherson-Pillay95} gave a general structural classification of primitive permutation groups of \fmrd, an analogue of the O'Nan-Scott theorem for finite groups \cite{Liebeck-Praeger-Saxl1988}. We formulate it here only for the most important case of  \emph{connected} primitive groups $G$.

\begin{fact} \emph{(Macpherson and Pillay} \cite{Macpherson-Pillay95}\emph{)} Let\/ $(G,X)$ be a connected
definably primitive permutation group of  \fmrd. Then one of the following holds:

\bi

\item[\emph{(1)}] $(G,X)$ is of \emph{affine type}, that is, contains a definable normal \emph{abelian} subgroup $S$ which acts on $X$ regularly. Moreover, $S$ is either an elementary abelian $p$-group for some prime $p$, or a torsion-free divisible group.

\item[\emph{(2)}]   $(G,X)$ is of \emph{regular type}, that is, contains a definable normal \emph{simple} subgroup $S$ which acts on $X$ regularly.

\item[\emph{(3)}] $G$ is \emph{almost simple}, that is, contains a definable normal \emph{simple} subgroup $T$, $T\leqslant G \leqslant \mathop{Aut} T$, and the stabiliser of a point $x\in X$ is non-trivial, $T_x \ne 1$.

\item[\emph{(4)}] $G\simeq T\times T$, the direct  product of two copies of a simple group $T$, and $(G,X)$ is equivalent to the action of $G$ on the coset space $G/D$ of the diagonal subgroup $D = \{\, (t,t), \; t \in T\,\}$.
\ei

\end{fact}

\subsection{What do primitive groups of affine type and regular type have in common?} \label{sec:common}

Let $(G,X)$ be a permutation group of one of these types. In the both cases,  $G$ contains a definable normal subgroup $S$  whose action on $X$ is sharply transitive (regular). If $M$ is a point stabiliser in $G$ of some $x \in X$, then $S \cap M=1$. Since the action is definably primitive, $M$ is a maximal definable subgroup in $G$, and therefore
\begin{quote}
\emph{the action of $M$ on $S$ by conjugation does not leave  any proper nontrivial definable subgroup $P$ of $S$ invariant,}
\end{quote}
for otherwise $M < PM < G$. It is natural to say in this situation that $S$ acts on $S$ \emph{irreducibly}. Obviously,  the action of a simple group $S$ on itself by conjugation is irreducible. Another obvious example of an irreducible action on a simple group $S$ is provided by a group $M = S \times L$ where $L$ is an arbitrary group acting on $S$ trivially, $[S,L] =1$. To take this situation into account we say that a group $M$ acts on a group $S$\emph{ faithfully} if $C_M(S) =1$.

It is easy to see that if $M$ acts on $S$ irreducibly then the action of $G = S\rtimes M$ on cosets of $M$ is definably primitive, and $(G,X)$ is equivalent to $(G, G/M)$. In particular, the action of $M$ on $X$ is equivalent to the action of $M$ on $S$. In particular, if $(G,X)$ is generically $n$-transitive, then $(M,S)$ is generically $(n-1)$-transitive.

An irreducible action of $M$ on $S$ imposes on $S$ severe  structural limitations.

\begin{lemma}[\cite{Macpherson-Pillay95}]  \label{lm:trichotomy-for-irreducible-actions}
If $M$ and $S$ are connected groups of finite Morley rank and $M$ acts on $S$ irreducibly, then {$S$} is either an abelian torsion free group, or an elementary abelian $p$-group, or simple.
\end{lemma}

The following configuration provides a key for unlocking the structure of the group $M$ if it acts on $S$ generically $n$-transitively.

Let $A$ be the generic subset of $S^n$ on which $M$ acts transitively. Then one can show that
$A$ is closed under taking inverses and permuting coordinates. We temporarily use the additive notation for the operation in $S$.

Consider  any  $(a_1,\ldots,a_n)$ in $A$. If one of the $a_i$'s is an involution, that is, $a_i =-a_i$, then its orbit under the action of $M$ is generic in $S$, and in that case an old classic lemma says that the connected group $S$ is an elementary abelian $2$-group (a more general statement is \cite[Proposition 1.2]{BBC-ID}). Leaving this special case aside, we see that  $(\pm a_1,\ldots,\pm a_n)$ and
$(a_{\sigma(1)},\ldots,a_{\sigma(n)})$  lie in $A$ for every $\sigma\in \symm_n$. By transitivity of the action, for each $1\leqslant i\leqslant n$ and $\sigma\in \symm_n$, there exist elements $e_i\in m$ and $h_\sigma\in M$ satisfying
$e_i(a_1,\ldots,a_n)=(a_1,\ldots,-a_i,\ldots, a_n)$,
$h_\sigma(a_1,\ldots,a_n)=(a_{\sigma(1)},\ldots,a_{\sigma(n)})$.

Set $H$ to be the setwise and $K$ be the pointwise stabilizer in $M$ of the set $\{\pm a_1,\ldots,\pm a_n\}$. Then $K\unlhd H$, and the following observation is easy.

\begin{lemma} \label{lm:existence-of-hyperoctahedral}  If $S$ is not an elementary abelian $2$-group, $H/K$ is isomorphic to the hyperoctahedral group $\symm_n\ltimes\mathbb Z_2^n$. In particular, if the action of $M$ on $S$ is  generically sharply
$n$-transitive, then $K=1$ and $H \simeq \symm_n\ltimes\mathbb Z_2^n$.

If $S$ is an abelian $2$-group then $H \simeq \symm_n$. \label{hyperoctahedral} \end{lemma}

\subsection{Why are primitive groups of affine type and regular type so different?} Because the latter could happen to be non-existent.

\begin{conjecture} \label{conj:definable-on-simple}
A connected group $M$  of   \fmrd\  cannot act  faithfully and irreducibly on a simple group {$S$} of \fmrd\ unless $M=S$ acting on itself by conjugation.
\end{conjecture}

Indeed this  would follow from the Cherlin-Zilber Algebraicity Conjecture.

Obviously, if Conjecture \ref{conj:definable-on-simple} is true, then the following is also true.

\begin{conjecture} \label{conj:no-regular-type}
Connected primitive permutation groups of \fmrd\  and regular type do not exist.
\end{conjecture}

This conjecture was stated, as a problem, in the Macpherson and Pillay paper \cite{Macpherson-Pillay95} of 1995): do primitive groups of regular type  (they were called primitive groups of type (2) there) exist?

The conjecture is still wide open, but some possible approaches to it can be found in Section \ref{sec:COnclusions}.

\subsection{Primitive groups of affine type} From now on $G$ is always a connected group of \fmr acting definably on a definable set $X$.

The following is the main result of the present paper. Its proof is spread over the series of papers \cite{bbgeneric,bbpseudo,bbsharp,bbnotsharp,bbsolvable,avb}.

\begin{theorem} \label{th:mail-on-primitive-affine}	Let $(G,X)$ be a connected definably primitive permutation group of affine type and assume that the action is generically $t$-transitive and $\rk(X) = n$.

Then $t \leqslant n +1$.

Moreover, if $t = n +1$, then  the group $G$ is isomorphic to the affine general linear group $\operatorname{\rm AGL}_n(K) \simeq K^n\rtimes \operatorname{\rm GL}_n(K)$  for some \acf $K$,  and the action of $G$ on $X$ is equivalent to $\operatorname{\rm AGL}_n(K)$  acting on the $n$-dimensional affine space $X=\mathbb{A}_n(K)$.
\end{theorem}

\begin{conjecture}\cite[Conjecture 5]{bordel}  \label{c:afterKnop}
{Let   $G$ be a connected  group of \fmr  with  $Z(G)$  finite and  $G/Z(G)$ simple. Assume that  $G$ acts definably and faithfully on an abelian group  $V$, and that $G$ is transitive on $V\setminus\{0\}$ and generically $2$-transitive. Then this action  is definably equivalent to the natural action of  $(\mathop{SL}_n(K)$ on $K^n)$ for some \acf $K$.}
\end{conjecture}

This is known for algebraic groups \cite[Satz~1]{Knop1983}, where the commutativity of $V$ is not required.

\subsection{A conjecture placing primitive groups of affine type into a wider classification}

Theorem \ref{th:mail-on-primitive-affine} suggests exact bounds for ranks of  all primitive groups of \fmrd.

\begin{conjecture} \label{conj:general bounds}
Let $(G,X)$ be a connected primitive permutation group  and $\rk X = n$.

Then \[\rk G \leqslant n(n+2).\]  Moreover,
\bi
\item[(a)] If $\rk G = n(n+2)$ then $G \simeq \operatorname{PGL}_{n+1}(K)$  for some \acf $K$, and the action of $G$ on $X$ is the natural action of $\operatorname{PGL}_{n+1}(K)$  on the projective space $\mathbb{P}_n(K)$.
\ei
In that case, $G$ is generically sharply $(n+2)$-transitive on $X$ and $b(G,X)= n+2$.

We further conjecture that there are only two other possibilities for the case $\rk G \geqslant n^2$ in which  $K$ is again an \acfd.

\bi
\item[(b)] $\rk G = n(n+1)$ and  $G \simeq \operatorname{AGL}_{n}(K)$.  The action of $G$ on $X$ is the natural action of  $\operatorname{AGL}_{n}(K)$ on the affine space $\mathbb{A}_n(K)$; $G$ is generically sharply $(n+1)$-transitive on $X$ and $b(G,X)= n+1$.

\item[(c)] $\rk G = n^2$ and $G \simeq \operatorname{ASL}_{n}(K)$.  The action of $G$ on $X$ is the natural action of  $\operatorname{ASL}_{n}(K)$ on the affine space $\mathbb{A}_n(K)$; $G$ is generically sharply $n$-transitive on $X$ and $b(G,X)= n$.
\ei
\end{conjecture}

The permutation groups appearing in the cases (b) and (c)  are of affine type; (b) is in the  focus  of the present paper. Also, it is the stabiliser of a point in the group of (a), $(\operatorname{PGL}_{n+1}(K), \mathbb{P}_n(K))$, while $\operatorname{GL}_{n}(K)$ is the stabiliser of a point in $(\operatorname{AGL}_{n}(K),\mathbb{A}_n(K))$  of case (b).

It is useful to compare this conjecture with the result by  Attila Maróti on orders of finite primitive groups.
\begin{fact} \cite{Maroti2002}
Almost all primitive permutation groups of degree $n$ have order at most
\[
n\cdot \prod_{i=0}^{[\log_2 n]-1} (n - 2^i ) < n^{1+[\log_2 n]},
\]
or have socle isomorphic to a direct power of some alternating group. The Mathieu groups,
$M_{11}$, $M_{12}$, $M_{23}$, and $M_{24}$ are the four exceptions.
\end{fact}

Together with Popov's theorem, Fact \ref{fact:Popov}, this leads to the following conjecture.

\begin{conjecture}
Almost all connected primitive permutation groups of \fmr  have maximum degree of generic transitivity at most $4$. The only exceptions are groups listed in Conjecture \ref{conj:general bounds}.
\end{conjecture}

\subsection{An equivalent form of our Main Theorem} \label{sec:equivalen-to-main}

As becomes clear from discussion in Section \ref{sec:common}, our Main Theorem can be equivalently formulated as the following statement.

\begin{theorem}	\label{th:real} \label{Vabelian}
Let $G$ be a connected group of \fmr acting on a connected abelian group $V$ of \fmrd, definably, faithfully and generically  $n$-transitively, where $n\geqslant \rk(V)$.

Then $n=\rk(V)$, $V\cong K^n$, $G\cong {\operatorname{GL}}_n({K})$, and the action $G\curvearrowright V$ is equivalent to $ {\operatorname{GL}}_n({K})\curvearrowright K^n$ for some \acf $K$.
\end{theorem}

This formulation has the advantage of a much wider range of techniques being accessible, including some ideas and results from representation theory. Notice that the word \emph{`primitive'} is no longer mentioned in the statement. It will be convenient to change notation and use the symbol $G\curvearrowright V$ for a faithful action of a group $G$ by automorphisms of an abelian group $V$.

In the next Section, we describe, step by step, the development of the proof of Theorem \ref{th:real}.

\section{Outline of the Proof of Theorem 3}

As explained in the previous sections, the main result of the present paper is Theorem~\ref{th:mail-on-primitive-affine}, which follows from Theorem \ref{th:real}.

In this section, we will sketch the proof of Theorem \ref{th:real} and some of its corollaries. {As mentioned before,} the proofs were published in a sequence of six papers, where different techniques were applied to  different aspects of the same structure; and naturally in each paper, different parts from the \fmr literature were useful. While outlining the proof, we will try to emphasize the milestones of this project.

\subsection{Trichotomy}  From the beginning we knew, from Angus Macintyre's structural results about abelian groups of \fmr \cite[Theorem 6.7]{bn}, that $V$ is either an elementary abelian $p$-group for some prime $p$, or a divisible torsion-free group. We decided to focus first on the case when
\bi
\item[Case 1.] $V$ is an elementary abelian $p$-group for an odd prime $p$,
\ei
and leave two other cases
\bi
\item[Case 2.] $V$ is an elementary abelian $2$-group,
\item[Case 3.] $V$ is a divisible torsion-free group.
\ei
for dessert, because of powerful results available in the last two cases;
more precisely, classification of groups of even type \cite{abc} in Case 2, and the Linearisation Theorem by James Loveys and Frank Wagner \cite{Loveys-Wagner1993}, \cite[Theorem A.20]{bn} in Case 3.

So in the next few sections we discuss our proof in Case 1, and return to Cases 2 and 3 in Section \ref{sec:char-0-2}.

\subsection{Basis of induction}
When working with groups $G$ and $V$ in the proof of Theorem \ref{th:real} one can do induction on the Morley rank, or rather, work with a hypothetical counterexample with the smallest possible Morley rank of the subgroup $V$.

In our context, cases $\rk(V)=1,2, 3$ are already known; we list them here.

An action on a group is called {\em minimal} if the only improper definable subgroups left invariant under this action are the finite subgroups.

\begin{fact} {\rm (Zilber} {\rm \cite[Theorem 9.5]{bn})} Assume that $G$ and $V$ are connected abelian groups of \fmrd, {$\rk(V)=1$},  and $G$ acts on $V$ faithfully. Then there exists an \acf $K$ such that the action $G\curvearrowright V$ is equivalent to the action $B\curvearrowright K^+$ for some subgroup $B$ in $K^*$.           \label{zilberaction}
\end{fact}

\begin{fact} {\rm (Deloro} {\rm \cite{Deloro09})} \label{deloro} Assume that $G$ and $V$ are connected groups of \fmr such that $G$ is non-solvable, $V$ is abelian of Morley rank $2$, and $G$ acts on $V$ faithfully and  minimally. Then there exists an \acf $K$ such that the action $G\curvearrowright V$ is equivalent to the action $\operatorname{GL}_2(K)\curvearrowright K^2$ or $\operatorname{SL}_2(K)\curvearrowright K^2$.
\end{fact}

\begin{fact} {\rm (Borovik and Deloro {\rm \cite{Borovik-Deloro}}, and Fr\'econ}   {\rm \cite{Frecon2018})}. \label{bordel} Assume that $G$ and $V$ are connected groups of \fmr such that $G$ is non-solvable, $V$ is abelian of Morley rank $3$, and  $G$ acts on $V$ faithfully and minimally.
	Then there exists an \acf $K$ such that the action $G\curvearrowright V$ is equivalent to the adjoint action $\operatorname{PSL_2}(K)\times Z(G)\curvearrowright K^3$ or to the natural action $\operatorname{SL_3}(K)* Z(G)\curvearrowright K^3$.
\end{fact}

\subsection{Chevalley groups, or How to identify a simple group of \fmr as a simple algebraic group?}

\

From the very beginning we were applying a general principle of project management:
\bq
\emph{identify your target and do the reverse analysis from the target to the present situation}.
\eq
As a rule, the final identification of a simple group $G$ of \fmr as an algebraic group is done by proving that $G$ is a Chevalley group (cf. \cite{GLS-3-1998,Malle-Testerman2011,Steinberg2016}) over an \acfd. We have not seen a single proof done by directly showing that $G$ is an algebraic variety over an \acf and that the group operations on $G$ are compatible with its structure as an algebraic variety. For example, the classification result for simple groups of even type \cite{abc} actually sounds as
\bq
A simple group of \fmr and even type is a Chevalley group over an \acf of characteristic $2$,
\eq
and therefore has, as a corollary, an identification theorem for algebraic groups in characteristic $2$:
\bq
A simple algebraic group over an \acf of characteristic $2$ is a Chevalley group.
\eq
So we needed to find a place for Chevalley groups in our project.

\subsection{Identification of the General Linear Group}

Since we took the  top-to-bottom  approach to our problem and expected the general linear group to appear in it, the first step was to prove some suitable identification theorems for the general linear group. So as a first step, we decided to obtain a version of Phan's theorem \cite{phan} for ${\operatorname{GL}}_n$ in the \fmr context.

{Why? Because} Phan's theorem (1970), which is a special case of Curtis--Tits type theorems for identification of Chevalley groups (discussed later), gives an identification of the finite special linear group $\operatorname{SL}_n(q)$ in the class of finite groups via their maximal tori and the corresponding root $\operatorname{SL}_2$-subgroups. In his proof, Phan constructed a $BN$-pair and showed that its Weyl group is isomorphic to the symmetric group $\symm_n$ \cite{phan}.

\begin{fact} {\rm (Phan \cite{phan})} Let $G$ be a finite group with subgroups $L_1,\ldots, L_n$ and $q$ be an odd prime power such that $q\neq 3$. Assume the following.
		
		\begin{itemize}[noitemsep,topsep=0pt,parsep=0pt,partopsep=0pt]	
			
			\item $G=\langle L_1,\ldots, L_n\rangle$.
			\item Each $L_k$ is isomorphic to $\sll_2(q)$.
			\item For every $1\leqslant k\leqslant n-1$, $\langle L_k, L_{k+1}\rangle$ is isomorphic to $\sll_3(q)$.	\item For every $1\leqslant k,l\leqslant n$ satisfying $|k-l|>1$, $L_k$ and $L_l$ commute. 	\item There exist elements $t_i, \in L_i$, $i = 1,\dots, n$, of order $q-1$ such that $\langle t_k,t_{k+1}\rangle$ is an abelian group of order $(q-1)^2$.
		\end{itemize}
		Then $G$ is isomorphic to a quotient of $\sll_{n+1}(q)$.
		
\end{fact}

A \fmr version of Phan's Theorem was proved by the first author in 1998. In the following statement, $\operatorname{(P)SL}_3(K)$ stands for a quotient group of $\operatorname{SL}_3(K)$ by a central subgroup.

\begin{our}\cite[Theorem 1.2]{berkmanthesis} Let $G$ be a simple group of \fmr which does not interpret any bad fields with subgroups $L_1,\ldots, L_n$ for $n\geqslant 3$.  Assume the following holds where $K$ is a fixed \acfd.
		
		\begin{itemize}[noitemsep,topsep=0pt,parsep=0pt,partopsep=0pt]	
			
			\item $G=\langle L_1,\ldots, L_n\rangle$.
			\item Each $L_k$ is isomorphic to $\sll_2(K)$.
			\item For every $1\leqslant k\leqslant n-1$, $\langle L_k, L_{k+1}\rangle$ is isomorphic to $\operatorname{(P)SL}_3(K)$.	
			\item For every $1\leqslant k,l\leqslant n$ satisfying $|k-l|>1$, $L_k$ and $L_l$ commute. 	\item There exist definable subgroups $A_k\leqslant L_k$ isomorphic to $K^*$ such that $\langle A_k,A_{k+1}\rangle\cong K^*\times K^* $.
			\item $\langle L_k,A_{k+1}\rangle=\langle L_k,A_{k-1}\rangle=\operatorname{GL}_2(K)$.
		\end{itemize}
		Then $G$ is isomorphic to  $\operatorname{PSL}_{n+1}(K)$.		
\end{our}

The technical assumption `does not interpret any bad fields' could be eliminated, in the specific context of proof of Theorem \ref{th:real}, with relative ease, but, by the time we started our project in 2011,  much more powerful tools had been developed, and we were happy to use them.

\subsection{Curtis-Tits Theorems}

By a Curtis--Tits type theorem, we mean an identification theorem of a group with a (quasi)simple Chevalley group by checking the `pairwise behaviour' of  its future root $\operatorname{SL}_2$-subgroups.  Richard Lyons  \cite[Section 29]{lyons} clarified the nature of these results as a statement about amalgams of groups and showed that they  were valid for infinite groups as well (see also  \cite[Section 5]{clinv}, \cite[Section 2.2]{lstar}). The first author interpreted Lyons's results from \cite[Section 29]{lyons} in the form convenient for application to groups of \fmr in \cite{clinv}.

\begin{our} \label{Lyons}   {\rm (Lyons \cite[Section 29]{lyons}, Berkman  \cite[Fact 5.2]{clinv})}
	Let\/ $G$ be a group,
	${F}$ an algebraically closed field and\/ $I$
	one of the connected Dynkin diagrams
	of the simple algebraic groups of Tits rank at least $3$.
	Let\/ $\{K_i:i\in I\}$ be a collection
	of subgroups of $G$ centrally
	isomorphic to $PSL_2({ F})$ indexed by $I$,
	and let\/ $T_i <  K_i$ be a maximal torus in
	$K_i$ for\/ $i\in I$.
	Also assume that the following statements hold.
	\begin{itemize}
		\item[{\rm (1)}] $G=\langle K_i\mid i\in I\rangle$.
		
		\item[{\rm (2)}] $[T_i,T_j]=1$ for all $i,j \in I$.
		
		\item[{\rm (3)}] $[K_i,K_j]=1$, if $i\ne j$ and $(i,j)$ is not an edge
		in $I$.
		
		\item[{\rm (4)}] $G_{ij} =\langle K_i,K_j\rangle$ is centrally isomorphic to
		$PSL_3({ F})$  if\/ $(i,j)$ is
		a single edge in $I$.
		
		\item[{\rm (5)}] $G_{ij}=\langle K_i,K_j\rangle$ is centrally isomorphic to
		$PSp_4({ F})\cong PSO_5({
			F})$  if $(i,j)$ is a double edge in $I$.
		Moreover, in that case, if $r_i$ and $r_j$ are involutions in
		$N_{K_i}(T_iT_j)$ and $N_{K_j}(T_iT_j)$, then the order of the element
		$r_ir_j$ in $N_{G_{ij}}(T_iT_j)/T_iT_j$ is $4$.
		
		\item[{\rm (6)}] $K_i$, $K_j$
		are root\/ $SL_2$-subgroups of\/ $G_{ij}$
		corresponding to the maximal torus $T_iT_j$ of\/ $G_{ij}$.
	\end{itemize}

	Then there is a homomorphism from the corresponding simply connected
	simple algebraic group $G'$ of type $I$ over $F$ onto
	$G$, under which the root $SL_2$-subgroups $K'_i$ for some
	simple root system of $G'$ correspond to the subgroups $K_i$.
	
\end{our}

In the \fmr context, it  was used by us in the proof of the following Generic Identification Theorem  developed by us for a different and more general project directed at the proof of the Cherlin-Zilber Algebraicity Conjecture.

\subsection{A Generic Identification Theorem}

Lyons's Theorem \ref{Lyons} was used in {the} proof of a theorem which  identifies `large' simple groups of \fmr with some specific conditions as Chevalley groups.

\begin{our} {\rm \cite[Theorem 1.2]{git}} Let $G$ be a simple  $K^*$-group of finite
	Morley rank, $p$ a prime, and $D$ a maximal $p$-torus  in $G$ of Pr\"{u}fer rank at least $3$. Assume that
	$G=\langle C^\circ_G(x) \mid x \in D, \;|x|=p\rangle$,
	and for every element $x$ of order
	$p$ in $D$, the group $C^\circ_G(x)$ is of $p'$-type and
	$C^\circ_G(x)=F^\circ(C^\circ_G(x))E(C^\circ_G(x))$.
	Then $G$ is a Chevalley group over an \acf of characteristic not $p$.
\end{our}

We explain terms used in this Theorem: a divisible abelian subgroup is called a {\em torus}.
A torus which is also a $p$-group is called a $p$-torus for
short. Note that a $p$-torus of finite Morley rank is a direct product of
finitely many copies of the Pr\"{ufer} $p$-group $C_{p^\infty}$. The number of copies of $C_{p^\infty}$
in the torus is called the Pr\"{u}fer rank of the torus.

	A group of finite Morley rank is said to be of {\em $p'$-type}, if
it contains no infinite abelian subgroup of exponent $p$.  A group \fmr whose infinite simple definable and connected sections are all algebraic groups over algebraically closed fields is called a {\em $K$-group}, and a group of \fmr whose  proper definable sections are all $K$-groups is called a {\em $K^*$-group}.

{There was one more Generic Identification Theorem, \cite{lstar}, also coming from Lyons's Theorem \ref{Lyons}, but it was not used in our project.}

\subsection{A version of a Generic Identification Theorem used in our project}

However, since the $K^*$-condition was difficult to check in our context, we first proved the following stronger version of the Generic Identification Theorem. We showed that the $K^*$-condition could be weakened to: proper definable subgroups  containing a fixed maximal $p$-torus  (for $p$ prime) are $K$-groups. Actually, we needed only the case $p=2$, as we hoped to build a large $2$-torus in $G$ from a hyperoctahedral subgroup arising from the multiple generic transitivity condition, see Lemma \ref{lm:existence-of-hyperoctahedral}.

\begin{our} \label{th:strongGIT} {\rm \cite[Theorem 1.1]{bbgeneric}}
	Let $G$ be a simple  group of finite
	Morley rank, $p$ a prime, and $D$ a maximal $p$-torus  in $G$	of Pr\"{u}fer rank at least $3$ such that
	every proper connected definable subgroup of $G$ which contains $D$ is a $K$-group. Assume that  $$G=\langle C^\circ_G(x) \mid x \in D, \;|x|=p\rangle,$$ and for
	every element $x$ of order
	$p$ in $D$, the group $C^\circ_G(x)$ is of $p'$-type and
	$C^\circ_G(x)=F^\circ(C^\circ_G(x))E(C^\circ_G(x))$.
	Then $G$ is a Chevalley group over an \acf of characteristic not $p$.	\end{our}

The proof involves a lot of technical details, so we give only a brief description here all the details can be found in \cite{bbgeneric}. Let $G$ be a simple group of \fmr with some extra technical conditions; fix a maximal $p$-torus $D$ in $G$. First, find the (future) root $\operatorname{SL}_2$-subgroups in $G$ normalized by $D$, then construct the Weyl group and show that it is one of the crystallographic reflection groups, and finally, apply the Lyons version of Curtis--Tits theorem mentioned above to conclude that $G$ is a Chevalley group over an \acfd.

However, this identification still was not enough to identify the general linear group in our problem. Since our project is about group actions,  using the concept of a pseudoreflection subgroup turned out to be more useful in identifying the group ${\operatorname{GL}}_n(K)$ and its natural action on $K^n$. This goes as follows.

\subsection{Pseudoreflection subgroups}

From now on, $K$ is an \acfd, $G$ and $V$ are connected groups of \fmrd, $V$ is abelian and $G$ acts on $V$ definably and faithfully.

\begin{definition}  A connected definable abelian subgroup $R$ of $G$ is called  a {\em pseudoreflection subgroup}
	if $V=[V,R]\oplus
	C_V(R)$, and $R$ acts transitively on the nonzero elements of
	$[V,R]$.  Moreover, {the \emph{pseudoreflection rank}} $\psrk(G)$ is the maximal number of pairwise commuting pseudoreflection subgroups in $G$.
\end{definition}

Pseudoreflection groups made their first appearance in \cite[Section III.1]{abc} and perhaps in related papers on simple groups of \fmr and even type that preceded the book.

As a first identification theorem of the general linear group, we proved the following, where the stronger version of the Generic Identification Theorem{, Theorem \ref{th:strongGIT},} was used.

\begin{our} {\rm \cite[Theorem 1.2]{bbpseudo}} \label{fact-pseudo}	Let $G$ be a connected group acting on a connected abelian
	group $V$ faithfully and irreducibly.
	If $G$ contains a pseudoreflection subgroup $R$ such that $\rk[V,R]=1$, and $\psrk(G)=\rk(V)$, then $G\curvearrowright V$ is equivalent to
	${\operatorname{GL}}_n({K})\curvearrowright K^n$ for some \acf $K$, where $n=\rk(V)$.\end{our}
\noindent
{\em Sketch of proof.}
Let $G$ be a counterexample  to the statement, of minimal Morley rank.
One can prove that the centralizers of non-central involutions in $G$ are direct sums of general linear groups, and also that $G/Z(G)$ is simple. After checking all the conditions of {Theorem \ref{th:strongGIT}}, one can conclude that $G$ is a quasisimple Chevalley group of Lie rank at least 3. However, in such groups centralizers of involutions are well-known, and in none of them,  all centralizers are direct sums of general linear groups. 	
This contradiction shows that there is no minimal counterexample to our theorem. \hfill $\Box$

\medskip

We obtained two more identifications of the general linear group as corollaries of the above theorem.

\begin{corollary} {\rm \cite[Corollary 1.3]{bbpseudo}} Let $G\ltimes K^n$ be a connected group of finite Morley rank, in which $G$ is definable and acts on $K^n$ preserving only the additive group structure {\rm (}but not necessarily the  $K$-vector space structure{\rm )}. If\/ $\operatorname{GL}_n(K)\leqslant G$, then $G=\operatorname{GL}_n(K)$.
\end{corollary}

Below, ${\rm Pr_2}(G)$ stands for the Pr\"{u}fer 2-rank of $G$.

\begin{corollary} {\rm \cite[Theorem 1.4]{bbpseudo}} \label{corollary-prufer} Let $G$ be a connected group acting on a connected abelian
	group $V$ faithfully and irreducibly.  If\/ ${\rm Pr_2}(G) = \rk(V)$,  then $G\curvearrowright V$ is equivalent to
	${\operatorname{GL}}_n({K})\curvearrowright K^n$ for some \acf $K$, where $n=\rk(V)$.
\end{corollary}

\subsection{Generically Sharply Multiply Transitive Actions}

As a first approximation to our goal, we studied groups with  a generically {\em sharply} multiply transitive action in \cite{bbsharp}. In this subsection, we will work under the hypothesis below. Note that since groups behave very differently over odd and even characteristics, we exclude the even characteristic case here by requiring that $V$ has no involutions.

\begin{hypothesis}
	$G$ is a connected group of \fmrd, $V$ is a connected abelian group of \fmr with no involutions, $G$ acts on $V$ definably, faithfully and generically sharply $n$-transitively, where
	$n\geqslant \rk(V)$.
\end{hypothesis}

A crucial observation is that $G$ has subgroups isomorphic to the hyperoctahedral group $\symm_n\ltimes\mathbb Z_2^n$, {Lemma \ref{lm:existence-of-hyperoctahedral}}. {The group} $\symm_n\ltimes\mathbb Z_2^n$ {is quite prominent in the theory of simple Lie groups and simple algebraic groups, it} is {the} Coxeter group of type $BC_n$, and it is the Weyl group of {orthogonal $\mathop{\rm SO}_{2n+1}$ and symplectic $\mathop{\rm Sp}_{2n}$ groups}.

\medskip
Now we are ready to prove the result for generic sharp multiple transitivity.

\begin{our} {\rm \cite{bbsharp}}	\label{sharp}
	Suppose that $G$ is a connected group of \fmrd, $V$ is a connected abelian group of \fmr with no involutions, $G$ acts on $V$ definably, faithfully and generically sharply $n$-transitively, where
	$n\geqslant \rk(V)$.
	Then $n=\rk(V)$, and $G\curvearrowright V$ is equivalent to
	${\operatorname{GL}}_n({K})\curvearrowright K^n$ for some \acf $K$ of characteristic not\/ $2$.
\end{our}

\noindent
{\em Sketch of the proof.} 	We will use induction on $n\geqslant 1$.	By {Lemma \ref{lm:existence-of-hyperoctahedral}},	we know that
$G$ contains copies of the hyperoctahedral group $\symm_n\ltimes\mathbb Z_2^n$.
With the above notation, set $U_i=[V,e_i]$;  then {$V=\bigoplus_{i=1}^n U_i$} and hence $\rk(U_i)=1$, $\rk(V)=n$ easily {follows}.   It is possible to show that
there exists a subgroup of $C_G(e_1)$ which acts generically sharply\ ${(n-1)}$-transitively on {$\bigoplus_{i=2}^n U_i$}. Hence, 	by induction, we have $\operatorname{GL}_{n-1}(K)$ in $G$, where $K$ is an \acf of characteristic not 2.    Now, repeating this with $C_G(e_n)$, one can
obtain a torus $(K^*)^n$  in $G$. 	Therefore,  $\Pr_2(G)=n=\rk(V)$, and thus we can apply   Corollary~\ref{corollary-prufer} and obtain the desired result. \hfill $\Box$

\subsection{The Sharpness of the Action}

In this subsection, we will prove that the action of $G$ on $V$ is generically {\em sharply} $n$-transitive under the following assumptions. For details, see \cite{bbnotsharp}.

\begin{hypothesis}
	$G$ is a connected group, $V$ is a connected elementary abelian $p$-group, where $p$ is an odd prime, $G$ acts on $V$ definably, faithfully and generically $n$-transitively, and
	$n\geqslant \rk(V)$.
\end{hypothesis}

By Lemma \ref{hyperoctahedral},  $G$ contains the hyperoctahedral group $\symm_n\ltimes\mathbb Z_2^n$ as a definable section $H/K \simeq \symm_n\ltimes\mathbb Z_2^n$ for some definable subgroups $1\leqslant K \lhd H \leqslant G$.

Perhaps the most technical result of this project is showing the triviality of $K$. In the proof, we used the following version of Maschke's Theorem, which is based on a result of Borovik \cite[Theorem 7]{avb}.

\begin{theorem} {\rm \cite[Theorem 2]{bbnotsharp}}  Let $V$ be a connected elementary abelian $p$-group of finite Morley rank, $X$ a finite $p'$-group acting on $V$ definably, and $R$ the enveloping algebra over $\mathbb F_p$ for the action of $X$ on $V$. Assume that  $A_1,A_2, \dots, A_m$ is the complete list of non-trivial simple  submodules for $R$ in $V$, up to isomorphism.
	Then	$\rk V \geqslant m$.
\end{theorem}

Groups of \fmr have good Sylow type theorems, for Sylow 2-subgroups and also for the so-called  \emph{maximal good tori} (see {\cite[Sections I.4.4, I.11.3, and IV.1.2]{abc})}, \ and this allowed us to use some Frattini type arguments for the  careful study of the structure of  $K$, first showing that $K$ was finite. After that, some delicate finite group theory arguments (we cannot give details here) were essential to reach the following conclusion.

\begin{proposition} If $V$ is an elementary abelian $p$-group, for odd prime $p$, then {within this setup  $K=1$}; that is, $G$ acts generically sharply $n$-transitively on $V$.
\end{proposition}

Therefore, we are able to remove the sharpness assumption from Theorem~\ref{sharp}.

\begin{our} {\rm \cite{bbnotsharp}}
	Suppose that $G$ is a connected group of \fmrd, $V$ is a connected elementary  abelian $p$-group of \fmrd, where $p$ is odd, $G$ acts on $V$ definably, faithfully and generically $n$-transitively, where
	$n\geqslant \rk(V)$.
	Then $n=\rk(V)$, and $G\curvearrowright V$ is equivalent to
	${\operatorname{GL}}_n({K})\curvearrowright K^n$ for some \acf $K$ of characteristic $p$. \label{notsharp}
\end{our}

\subsection{Characteristic 0} \label{sec:char-0-2}

In the previous section, even though some results are valid in a more general context, the final statement (Theorem \ref{notsharp}) is proved only when $V$ is an elementary abelian $p$-group, for odd $p$. When $V=K^n$ for some \acf $K$, obviously, $V$ may be an elementary abelian 2-group, or a divisible {torsion-free} abelian group. We will treat these two cases in this section  and the next, and show that there is no other case  in section \ref{sec:proof-thm-3}.

When $V$ is divisible abelian, we use the following result by Loveys and Wagner {(in a slightly simplified version)}.

\begin{fact} {\rm \cite{Loveys-Wagner1993}, \cite[Theorem A.20]{bn}}\label{loveyswagner}
	Let $G$ be an infinite group  {of \fmr} acting on an infinite divisible torsion-free abelian group $V$  {of \fmr definably}. If the action is faithful and irreducible, then there exists an \acf $F$ of characteristic 0 such that $V$ is a vector space over $F$, $G$ is definably isomorphic to a subgroup $H$ of $\operatorname{GL}(V)$, and the action  $G\curvearrowright V$ is equivalent to the action  $H\curvearrowright V$.
\end{fact}

Now we are ready to sketch the proof of the following.

\begin{proposition} {\rm \cite[Lemma 4.2]{bbsolvable}}
	Suppose that $G$ is a connected group of \fmrd, $V$ is a connected {torsion-free} divisible abelian, $G$ acts on $V$ definably, faithfully and generically $n$-transitively, where
	$n\geqslant \rk(V)$.
	Then $n=\rk(V)$, and $G\curvearrowright V$ is equivalent to
	${\operatorname{GL}}_n({K})\curvearrowright K^n$ for some \acf $K$ of characteristic $0$. \label{char0}
\end{proposition}

\noindent
{\em Sketch of proof.}  Proof is by induction on $\rk(V)\geqslant 1$. If there exists a  non-trivial proper definable subgroup $U\leqslant V$ fixed by $G$, then the induction hypothesis can be applied to $G \curvearrowright V/U$, and hence $\rk U=0$ is obtained. Therefore the action is irreducible and Fact~\ref{loveyswagner} is applicable. The rest follows from rank computations.
\hfill $\Box$

\medskip

\subsection{Characteristic 2} \label{sec:char-2}

The characteristic $2$ case requires two important results: an expanded version of Borovik's Linearisation Theorem \cite[Theorem 4]{avb},
{and the} classification of simple groups of \fmr and of even type.  We copy both results verbatim from \cite{bbsolvable}.

\begin{our}
	{\rm \cite[Theorem 2]{bbsolvable}} \label{linearisation} Let $p$ be a prime, $V \rtimes G$ a group of finite Morley rank, where $V$ is a connected elementary abelian $p$-group of finite Morley rank, and $G$ is a connected group of finite Morley rank which acts on $V$ faithfully, definably, and irreducibly. Assume that $L \lhd G$ is a definable normal subgroup isomorphic to a simple algebraic group over an algebraically closed field $F$ of characteristic $p$.
	
	Then
	\begin{itemize}
		\item The abelian group $V$ has a structure of a finite dimensional   $F$-vector space, and  the action of $G$ on $V$ is $F$-linear.
		\item The group $G$ can be decomposed as a central product
		\[
		G = L \ast M_1\ast \cdots \ast M_k \ast T
		\]
		where $M_i$, $i = 1, \dots, k$ are definable and isomorphic to simple algebraic groups over $F$ {\rm (}possibly $k=0${\rm )}, and is a torus (possibly $T=1$).
		\item The enveloping ring $R=R_V(G)$ additively generated by $G$ in $\mathop{{\rm End}} V$ is definable and  equal to the algebra $\mathop{{\rm End}}_F V$.
	\end{itemize}
\end{our}

{In the next result,} $O_2 (G)$ denotes the maximal
	normal 2-subgroup of $G$, {and $O_2^\circ(G) = (O_2 (G))^\circ$ is its connected component}.

\begin{fact}{\rm \cite[Proposition X.2,  {page 502}]{abc}} \label{centralproduct} Let $G$ be a connected group of finite Morley rank
	containing no nontrivial $2$-torus. Then $O_2(G) = O_2^\circ(G)$ and it is a definable
	unipotent subgroup of $G$. The quotient $G/O_2^\circ(G)$ has the form $Q\ast S$ with $S$
	a central product of quasisimple algebraic groups over algebraically closed
	fields of characteristic 2, and $Q$ a connected group without involutions.
\end{fact}

Finally we are ready to treat the characteristic 2 case.

\begin{proposition} {\rm \cite[Proposition 4.3]{bbsolvable}}
	Suppose that $G$ is a connected group of \fmrd, $V$ is an elementary abelian $2$-group, $G$ acts on $V$ definably, faithfully and generically $n$-transitively, where
	$n\geqslant \rk(V)$.
	Then $n=\rk(V)$, and $G\curvearrowright V$ is equivalent to
	${\operatorname{GL}}_n({K})\curvearrowright K^n$ for some \acf $K$ of characteristic $2$. \label{char2}
\end{proposition}

\noindent
{\em Sketch of Proof.} By studying the action of the subgroup $O_2(G)$ on $V$, we show that it is  trivial; hence Fact~\ref{centralproduct} applies and we conclude that $G=Q*S_1*\cdots *S_k$, where $Q$ is a connected
	group with no involutions, and $S_i$'s are  quasisimple algebraic groups. Then we apply Theorem~\ref{linearisation} to the action of $G$ on $\bar{V}=V/C_V(G)$ for the case $p=2$, and we get an algebraically closed field $K$ of characteristic 2, over which $\bar{V}$ is a vector space, and  $G$ is isomorphic to $\operatorname{GL}(\bar{V})$. By further analysis, we conclude that $C_V(G)$ is trivial, and hence the named actions in the statement are equivalent. \hfill $\Box$

\subsection{Proof of Theorem~\ref{Vabelian}} \label{sec:proof-thm-3}

{We are ready to prove Theorem~\ref{Vabelian}.} First, we need the following classical result of Macintyre.

\begin{fact} {\rm \cite[Theorem 6.7]{bn}}
	Let $G$ be an abelian group of \fmrd, then  $G=B\oplus D$, where $B$ is a subgroup of bounded exponent and $D$ is  {a definable} divisible subgroup.
	\label{macintyre}
\end{fact}

\noindent
{\em Sketch of proof of Theorem~\ref{Vabelian}.} Again we use induction on $\rk(V)$.  First we express $V=B\oplus D$ by
using Macintyre's Theorem. Then induction hypothesis gives us either $V=D$ is divisible, or $V=B$ is of bounded exponent. It is not hard to show that the exponent of $V$ is prime in the latter case. Hence, by using the above results, we complete the proof.
\hfill $\Box$

{As was explained in section \ref{sec:equivalen-to-main}, the main theorem of this work (Theorem~\ref{th:mail-on-primitive-affine}) follows from Theorem~\ref{Vabelian}}.

\subsection{A slight generalisation}

{Finally, we will  generalise Theorem~\ref{Vabelian} to the case where $V$ is a connected solvable group of \fmrd.}

\begin{corollary} \label{cor:Vsolvable}	\cite{bbsolvable} Suppose that $G$ is a connected group of \fmrd, $V$ is a connected solvable group, $G$ acts on $V$ definably, faithfully and generically $n$-transitively, where
	$n\geqslant \rk(V)$.
	Then $n=\rk(V)$, and $G\curvearrowright V$ is equivalent to
	${\operatorname{GL}}_n({K})\curvearrowright K^n$ for some \acf $K$.
\end{corollary}

\noindent
{\em Sketch of proof.}    Induction is used on $\rk(V)$. It suffices to show $V$ is abelian. Note that $G$ acts  generically {$n$}-transitively on $V/[V,V]$. By the induction hypothesis, $\rk(V/[V,V])=n\rk(V)$ and thus $[V,V]=0$. Therefore, we are done. \hfill $\Box$

\begin{problem} Generalise Corollary \ref{cor:Vsolvable} to the case when $V$ is an arbitrary connected group of \fmrd. The assumption that $V$ is solvable appears to be exessive.

\end{problem}

\section{Conclusions and some open problems} \label{sec:COnclusions}

\subsection{Lessons learned} In this survey, we were trying to demonstrate that theory of groups of \fmr reached a stage in its development where a further progress requires some {serious} strategic planning. In the next section, we suggest some possible approaches to {the} analysis of primitive groups of regular type (see section \ref{sec:Macpherson-Pillay} for a definition). We wish to emphasise: this is not yet a plan, it is just a first sketch of what potentially could become a plan. We would be happy to have a discussion with anyone who could {be} interested.

\subsection{Possible approaches to Conjectures \ref{conj:definable-on-simple} and \ref{conj:no-regular-type}} \label{subsec:approaches-to-regular} The first exploratory steps to the confirmation of Conjectures \ref{conj:definable-on-simple} and \ref{conj:no-regular-type}  could be solving the following problems.

First of all, we need to check the usual suspects, hypothetical \emph{bad groups}. An infinite simple group of \fmr is called \emph{bad} if every  proper definable subgroup is nilpotent. If such groups exist, they would be counterexamples to the Cherlin-Zilber Algebraicity Conjecture.

\begin{problem} \label{prob:noirreducibleonbad}
Prove that bad groups of \fmr do not admit connected groups of \fmr acting on them faithfully, definably, and irreducibly other than themselves  acting by conjugation.
\end{problem}

Here is a possible approach to this problem. It is known that bad groups do not admit involutive automorphisms \cite[Proposition 13.4]{bn}, so if $S$ is bad and a connected group $M$ acts on $S$ faithfully and irreducibly, then $M$ contains no involutions. Also, maximal definable subgroups of a bad group (traditionally called \emph{Borel subgroups}) are connected, nilpotent, and conjugate. Moreover, they are \emph{disjoint}: if $B$ and $C$ are Borel subgroups in $S$ and $B \ne C$, then $B \cap C=1$. Set $G= SM$, then the Frattini argument yields a factorisation $G = S\cdot N_G(B)$. But we also have the factorisation $G= SM$, which creates some tension, which perhaps could be exploited and lead to a contradiction.

The following issue is perhaps considerably harder.

\begin{problem}
Prove that simple groups of \fmr and degenerate type do not admit connected groups of \fmr acting on them faithfully, definably,  and irreducibly  other than themselves  acting by conjugation.
\end{problem}

The same problem can be approached from an opposite direction.

\begin{problem}
Prove that  a simple algebraic group over an \acf $K$ cannot act definably and irreducibly on a simple group of \fmrd\ other than itself.
\end{problem}

The case when $K$ is of characteristic $2$ is likely to be easy in view of Tuna Alt\i nel's Lemma \cite[Proposition I.10.13]{abc}: connected nilpotent $2$-groups {of bounded exponent} can act only trivially on groups without involutions.

\begin{problem} \label{prob:boundson-t}
Let $(G,X)$ be a primitive permutation group of \fmrd\  and regular type which decomposes as $G=SM$ where $S$ is a regular normal {subgroup} and $M$ a point stabiliser. Assume that $M$ is generically $t$-transitive on $S$. Get some bounds on $t$ which do not depend on $G=SM$.
\end{problem}

First of all, the situation when $S$ has Morley rank $\leqslant 4$ is well understood. It is well-known that there is no simple group of Morley rank $1$ or $2$ \cite{Ch-GSR}. A simple group of Morley rank $3$ is isomorphic to $\operatorname{PSL}_2(K)$ for some \acf $K$, by Oliver Fr\'{e}con \cite{Frecon2018}.  A simple group of Morley rank $4$ is bad by Joshua Wiscons \cite[Theorem A]{wiscons2016} and therefore does not admit involutive automorphisms, which means that we are in a very restricted case of Problem \ref{prob:noirreducibleonbad}.

Regarding potential bounds on the degree of generic multiple transitivity asked for in Problem \ref{prob:boundson-t}, proving the bound $t\leqslant 5$ appears to be feasible. By Lemma \ref{hyperoctahedral},  $M$ contains the hyperoctahedral group $\symm_t\ltimes\mathbb Z_2^t$ as a definable section $H/K \simeq \symm_t\ltimes\mathbb Z_2^t$ for some definable subgroups $1\leqslant K \lhd H \leqslant M$. Now  assume that $t \geqslant 6$. The presence of the hyperoctahedral section $\symm_t\ltimes\mathbb Z_2^t$ with $t\geqslant 6$ in $M$ is likely to mean that $M$  is a group of odd type of Pr\"{u}fer $2$-rank at least $3$, where the machinery developed for analysis of odd type groups works quite well. The irreducibility of the action of $M$ on $S$ is a powerful property which so far had not been properly explored. In a closely related development,   Tuna Alt\i nel and   Joshua Wiscons \cite{aw2022} studied the actions of the alternating groups $\mathop{\rm Alt}_t$ on groups $S$ of degenerate type, and obtained some promising results; what would happen if this  $\mathop{\rm Alt}_t$ lives inside a connected group $M$ acting on $S$ irreducibly?

So we have an umbrella problem:

\begin{problem}
Systematically study faithful irreducible definable actions of connected groups $M$ of \fmr on simple groups $S$. Pay special attention to situations when $M$ contains large alternating and especially hyperoctahedral subgroups.
\end{problem}

\section*{Acknowledgements}

{We thank the organizers of the 14th Ukraine Algebra Conference (Sumy Conference) which took place between 3--7 July 2023 for giving us an opportunity to present our work. We express our solidarity with Ukrainian mathematicians.}

Gregory Cherlin has carefully read  the paper and suggested quite a number of subtle stylistic changes and grammar improvements.

\end{document}